\documentclass[12pt,a4paper]{amsart}
\usepackage[english]{babel}
\usepackage[left=2cm,right=2cm,top=2cm,bottom=2cm]{geometry}
\usepackage[numbers,sort&compress]{natbib} 
\usepackage{enumitem}
\usepackage{hyperref}
\usepackage{mathtools}
\usepackage{esint}
\usepackage{cases}

\theoremstyle{definition}
\newtheorem{defi}{Definition}[section]
\newtheorem{rem}[defi]{Remark}
\newtheorem*{notation}{Notation}
\theoremstyle{plain}
\newtheorem{teo}[defi]{Theorem}
\newtheorem{lem}[defi]{Lemma}
\newtheorem{pro}[defi]{Proposition}

\newcommand{\N}{\mathbb{N}}
\newcommand{\vp}{\varphi}
\newcommand{\R}{\mathbb{R}}
\newcommand{\pp}{\partial}
\newcommand{\se}{\subseteq}
\newcommand{\ceq}{\coloneqq}
\newcommand{\pxu}{{\mathcal{P}^X_{\mathcal{J}_u}}}
\newcommand{\res} {\mathop{\hbox{\vrule height 7pt width .5pt depth 0pt \vrule height .5pt width 6pt depth 0pt}}\nolimits}
\newcommand{\leb}{\mathcal{L}}
\newcommand{\sbv}{\operatorname{SBV}}
\newcommand{\bv}{\operatorname{BV}}
\newcommand{\J}{\mathcal{J}}
\newcommand{\hau}{\mathcal{H}}
\newcommand{\nl}{\left\|}
\newcommand{\nr}{\right\|}
\newcommand{\loc}{\operatorname{loc}}
\newcommand{\ap}{{\operatorname{ap}}}
\newcommand{\s}{\mathcal{S}}
\newcommand{\p}{{\mathfrak{p}}}

\author[M. Di Marco]{Marco Di Marco}
\address{ETH Z\"urich, Department of Mathematics, R\"amistrasse 101, 8092 Z\"urich, Switzerland.\newline \indent 
Dipartimento di Matematica ``T.~Levi-Civita'', Università di Padova, via Trieste 63, 35121 Padova, Italy. }
\email{mdimarco@ethz.ch}

\title[Intr. perimeter, compactness and Poincaré ineq. for SBV functions in CC spaces]{Intrinsic perimeter, compactness and Poincaré inequality for SBV functions in Carnot-Carathéodory spaces}



\thanks{The author warmly thanks Sebastiano Don and Davide Vittone for many precious discussions and for an initial reading of this paper.  The author is supported by SNSF Starting Grant TMSGI2\textunderscore226018, University of Padova and GNAMPA of INdAM. Part of this work was written while the author was a guest of the Forschungsinstitut f\"ur Mathematik (FIM) in Z\"urich: he wishes to thank the FIM for the support as well as for the pleasant and exceptionally stimulating atmosphere. }

\begin{document}

\begin{abstract}
By introducing an intrinsic perimeter measure for intrinsic countably rectifiable sets, we prove a compactness result and a Poincaré inequality for special functions with bounded variation in equiregular Carnot-Carathéodory spaces which satisfy an additional natural assumption, called property $\mathcal R$. 
\end{abstract}

\maketitle

\section{Introduction}

One of the aims of this paper is to expand the theory of special functions with bounded variation in Carnot-Carathéodory spaces by proving the sub-Riemannian counterparts of two milestones of the theory of classical Euclidean $\sbv$ functions, namely the compactness theorem and the Poincaré inequality. In order to do so, we need to introduce an intrinsic perimeter measure for intrinsic countably rectifiable sets, which allows us to work in a natural framework and is of independent interest.

Before stating our main results, let us briefly recall the definitions of Carnot-Carathéodory spaces (CC spaces) and (special) functions with bounded variation ($\bv_X$, $\sbv_X$ functions). We refer the reader to Section \ref{sec_prel} for precise definitions.

A {\em Carnot-Carathéodory space} (see Definition \ref{def_cc}) is the space $\R^n$ endowed with a distance arising from a collection $X=(X_1,\dots,X_m)$ of smooth and linearly independent vector fields satisfying the H\"ormander condition. In this paper, we will deal with {\em equiregular} CC spaces, where a homogeneous dimension $Q$, usually larger than the topological dimension $n$, can be defined.

The space $\bv_X(\Omega)$~\cite{cdg,fssc} of functions with bounded $X$-variation (see Definition \ref{def_bvx}) consists of those functions $u$ on an open set $\Omega\se (\R^n,X)$ whose derivatives $X_1u, \dots, X_mu$ in the sense of distributions are represented by a vector-valued measure $D_Xu$ with finite total variation $|D_Xu|$. In the past few years a lot of effort was put into the study of $\bv_X$ functions: see for instance \cite{bu95,dgn98,fgw94,fssc01,fssc03,gn96,garofalonhieu,CapGarAhlfors,DanGarNhi,Selby} and the more recent \cite{agm15,am03,as10,bmp12,cm20,dmv19,magnani02,marchi14,sy03,vittone2012,dontesi,dv,dv19,sbvx}.

In \cite{dv} it was proven that, if $u \in \bv_X(\Omega)$, then $D_Xu$ can be decomposed as
\[
D_Xu=D^{\ap}_Xu\: \leb^n+D^s_Xu=D^{\ap}_Xu\: \leb^n+D^j_Xu+D^c_Xu,
\]
where $D^\ap_Xu$ is the approximate $X$-gradient of $u$, $\leb^n$ is the usual Lebesgue measure, $D_X^su$ is the singular part of $D_Xu$, $D^j_Xu$ is the jump part of $D_Xu$, and $D_X^cu$ is the Cantor part of $D_Xu$.

The space of $\sbv_X$ functions was introduced in \cite{sbvx}, where the following definition was given.

\begin{defi}\label{def_introsbvx}
 Let $\Omega\se\R^n$ be an open subset of an equiregular Carnot-Carathéodory space $(\R^n,X)$ and let $u \in \bv_X(\Omega)$. We say that $u$ is a \emph{special function of bounded $X$-variation}, and we write $u \in \sbv_X(\Omega)$, if
\begin{enumerate}
    \item[(i)] $D_X^c u=0$, and
    \item[(ii)] the jump set $\J_u$ of $u$ is a countably $X$-rectifiable set\footnote{Recall that a set is said to be countably $X$-rectifiable if it can be covered, up to a set which is negligible with respect to the Hausdorff measure $\hau^{Q-1}$, by a countable family of $C^1_X$-hypersurfaces (see Definition \ref{def_ipersup}).}.
\end{enumerate}
\end{defi}
However, in this paper we will work under the assumption that the CC space $(\R^n,X)$ satisfies the so-called property $\mathcal R$ (``rectifiability'', see Definition \ref{def_pror}) which, as proved in \cite[Theorem 1.5]{dv}, implies that condition (ii) in Definition \ref{def_introsbvx} is automatically satisfied for every $u \in \bv_X$. 
Let us stress that there is a plethora of examples of CC spaces which satisfy  property $\mathcal{R}$, such as Heisenberg groups, step 2 Carnot groups and Carnot groups of type $\star$, see \cite[Theorem 4.3]{dv}.

In Section \ref{sec_intrper} we present the first contribution of this paper: the definition of the \emph{intrinsic perimeter measure} for countably $X$-rectifiable sets (see Definition \ref{def_pxu}).
\begin{defi}
We define, for every countably $X$-rectifiable set $R \se (\R^n,X)$, the \emph{intrinsic perimeter measure} $\mathcal P^X_R$ as the unique measure on $R$ such that, for every $C^1_X$-hypersurface $S\se (\R^n,X)$, one has
\[
\mathcal P^X_R \res S =P^X_S \res R
\]
where $P^X_S$ is the (local) $X$-perimeter measure (see Definition \ref{def_xper}) of (each of) the two $X$-regular open sets (see \cite{vittone2012}) in which (locally) $S$ separates $(\R^n,X)$. By \cite{AmbrosioAhlfors}, the intrinsic perimeter measure $\mathcal P^X_R$ is also absolutely continuous with respect to $\hau^{Q-1}$.
\end{defi}
By using the intrinsic perimeter measure, one can reformulate some results from \cite{sbvx} in a more natural way, as the proposition below shows.

\begin{pro} Let $\Omega$ be an open subset of an equiregular Carnot-Carathéodory space $(\R^n,X)$ satisfying property $\mathcal R$. The following statements are equivalent:
    \begin{enumerate}
        \item[(i)] $u\in \sbv_X(\Omega)$;
        \item[(ii)] $u\in \bv_X(\Omega)$ and  
        \[
        D_Xu= D^\ap_Xu \leb^n+(u^+-u^-)\nu_{\J_u}\pxu
        \]
    \end{enumerate}
    where $u^\pm$ are the traces of $u$ on $\J_u$ (see Definition \ref{def_approxjump}) and $\nu_{\J_u}$ is the \emph{horizontal normal} of $\J_u$ (see Definition \ref{def_ipersup}). Observe that $\pxu$ does not depend on the function $u$ but only on the countably $X$-rectifiable jump set $\J_u$. 
\end{pro}

Having introduced the intrinsic perimeter measure $\pxu$, which is the \emph{natural} choice\footnote{Opposed to the Hausdorff measure, used in the classical Euclidean case.} to use in our sub-Riemannian generalizations, we can finally move to the presentation of the compactness theorem (Theorem \ref{teo_introcompact}) and the Poincaré inequality (Theorem \ref{teo_intropoincare})

The compactness theorem for classical $\sbv$ functions was first proved by L. Ambrosio in \cite{ambrosio89}. Later, several different approaches were developed: see for instance \cite{ambrosio95,albertimantegazza}. Our generalization of the compactness theorem for $\sbv_X$ functions is the following. 

\begin{teo}\label{teo_introcompact}
Let $\Omega$ be an open bounded subset of an equiregular Carnot-Carathéodory space $(\R^n,X)$ satisfying property $\mathcal R$. Let $(u_h)_{h \in \N}$ be a sequence of functions in $\sbv_X(\Omega)$, and assume that 
    \begin{itemize}
        \item[(i)] the functions $u_h$ are uniformly bounded in the $\bv_X$ norm (i.e., they are relatively compact with respect to the weak-$*$ topology of $\bv_X(\Omega)$\footnote{Recall that we say that $u_h$ converges to $u$ in the weak-$*$ topology of $\bv_X(\Omega)$ if $u_h\to u$ in $L^1(\Omega)$ and $\int_\Omega \vp dD_Xu_h \to \int_\Omega \vp dD_Xu$ for every $\vp \in C_c(\Omega)$.}),
        \item[(ii)] the approximate $X$-gradients $D^\ap_Xu_h$ are equi-integrable (i.e., they are relatively compact with respect to the weak topology of $L^1(\Omega,\R^m)$),
        \item[(iii)] there exists a function $f:[0,+\infty) \to [0,+\infty]$ such that 
        \[
        \lim_{t \to 0} \frac{f(t)}{t}\to +\infty,
        \]
        and
        \[
        \int_{\J_{u_h}}f(|u^+_h-u^-_h|)d\mathcal{P}^X_{\J_{u_h}} \leq C <+\infty.
        \]
    \end{itemize}
    Then, up to a subsequence, $(u_h)_{h \in \N}$ converges to a certain $u \in \sbv_X(\Omega)$. Moreover, the absolutely continuous part and the jump part of the derivatives converge separately in the weak\nobreakdash-$*$ topology of measures\footnote{Recall that we say that $\mu_h$ converges to $\mu$ in the weak-$*$ topology of measures if $\int_\Omega \vp d \mu_h \to \int_\Omega \vp d\mu$ for every $\vp \in C_c(\Omega)$.}.
\end{teo}

The proof of Theorem \ref{teo_introcompact} is provided in Section \ref{sec_compactness}. We follow the approach of G. Alberti and C. Mantegazza (\cite{albertimantegazza}), i.e., we prove the compactness through the use of a chain rule and a characterization of the singular part of the derivatives of $\bv_X$ functions, which are also new in the setting of CC spaces.

The Poincaré inequality for classical $\sbv$ functions was proved by E. De Giorgi, M. Carriero and A. Leaci in \cite{dgcl}. In order to prove our generalization, Theorem \ref{teo_intropoincare}, we follow their approach, which relies on the use of a chain rule, the coarea formula and the isoperimetric inequality.

\begin{teo}\label{teo_intropoincare}
 Let $K$ be a compact subset of an equiregular Carnot-Carathéodory space $(\R^n,X)$ satisfying property $\mathcal R$. Then there exist constants $C_K>0$, $R_K>0$ such that for every $p \in K$, $r \in (0,R_K)$ and $u \in \sbv_X(B(p,r))$ with $\pxu(\J_u)$ sufficiently small one has
    \[
   \left( \int_{B(p,r)} |\bar u-m|^\frac{Q}{Q-1}d\leb^n \right)^\frac{Q-1}{Q}\leq 2 C_K \int_{B(p,r)}|D^\ap_Xu|d\leb^n,
    \]
    where $\bar u$ is a specific truncation\footnote{I.e., a function defined as $\bar u \ceq \max\{ \min\{u,\tau_2\},\tau_1\}$ for some precise choices of $\tau_1,\tau_2 \in \R$.} of $u$ and $m$ is any median for $u$ in $B(p,r)$.
\end{teo}
 The above constants $C_K,R_K$ are coming from the isoperimetric inequality in CC spaces (Theorem \ref{teo_isoineq}) which, as said before, is pivotal in the proof of Theorem \ref{teo_intropoincare}. Such proof, together with the precise statement of Theorem \ref{teo_intropoincare}, is provided in Section \ref{sec_poinc}.

\section{Notation and preliminary results}\label{sec_prel}

\begin{defi}\label{def_cc}
Let $1\leq m \leq n$ be integers and let $X=(X_1,\dots,X_m)$ be a $m$-tuple of smooth and linearly independent vector fields on $\R^n$. We say that an absolutely continuous curve $\gamma\colon[0,T] \to \R^n$ is an \emph{$X$-subunit path} joining $p$ and $q$ if $\gamma(0)=p$, $\gamma(T)=q$ and there exist $h_1,\dots h_m \in L^\infty([0,T])$ such that $\sum_{j=1}^m h_j^2 \leq 1$ and 
\[
\gamma'(t)=\sum_{j=1}^m h_j(t)X_j(\gamma(t)), \quad \text{for a.e.\ $t\in [0,T]$.}
\]
For every $p,q \in \R^n$ we  define  
\[
d(p,q) \ceq \inf \lbrace T>0: \text{ there exists an $X$-subunit path $\gamma$ joining $p$ and $q$}\rbrace,
\]
where we agree that $\inf \emptyset \ceq +\infty$.

By the Chow–Rashevskii Theorem, if for every $p \in \R^n$ the linear span of all iterated commutators of the vector fields $X_1,\dots, X_m$ computed at $p$ has dimension $n$ (i.e. $X_1,\dots,X_m$ satisfy the \emph{H\"ormander condition}), then $d$ is a distance: the latter means that for every couple of points of $\R^n$ there always exists a $X$-subunit path joining them. In this case we say that $(\R^n,X)$ is a \emph{Carnot-Carathéodory space} of \emph{rank} $m$ and $d$ is the associated \emph{Carnot-Carathéodory distance}.         

For every $p \in \R^n$ and for every $i \in \N$ we denote by $\mathfrak{L}^i(p)$ the linear span of all the commutators of $X_1,\dots,X_m$ up to order $i$ computed at $p$. We say that a Carnot-Carathéodory space $(\R^n,X)$ is \emph{equiregular} if there exist natural numbers $n_0,n_1,\dots,n_s$ such that 
\[
0=n_0<n_1<\cdots<n_s=n \text{ and } \dim  \mathfrak{L}^i(p)=n_i, \quad \forall p \in \R^n, \forall i\in\{1,\dots, s\}.
\]
 The natural number $s$ is called \emph{step} of the Carnot-Carathéodory space. If $(\R^n,X)$ is equiregular, then the \emph{homogeneous dimension} is $Q \ceq \sum_{i=1}^si(n_i-n_{i-1})$.
\end{defi}

\begin{notation}
 In the following, $(\R^n,X)$ denotes an equiregular Carnot-Carathéodory space associated with the family $X=(X_1,\dots,X_m)$ satisfying property $\mathcal R$ (see Definition \ref{def_pror}). We use $d$ to denote the Carnot-Carathéodory distance associated with $X$, $B(\cdot,\cdot)$ to denote the associated open balls, $\hau^k$ to denote the associated Hausdorff $k$-measure and  $\leb^n$ to denote the usual Lebesgue measure. By $\Omega \se (\R^n,X)$ we denote a fixed open set and by $Q$ we denote the homogeneous dimension of $(\R^n,X)$. Later we will also use the following notation: for every $1 \leq i \leq m$ and $x \in \R^n$ we write $X_i(x)=(a_{i,1}(x),\dots,a_{i,n}(x))$ where $a_{i,t} \in C^\infty(\R^n)$ for $1 \leq t \leq n$. We also denote by $X_i^*$ the formal adjoint of $X_i$, i.e., for every $\vp \in C^1(\Omega)$, $x \in \Omega$ we write
\[
(X_i^*\vp)(x) \ceq \sum_{t=1}^n \frac{\pp (a_{i,t}\vp)}{\pp x_t}(x).
\]
\end{notation}
\begin{defi}\label{def_bvx}
We say that $u \in L^1_{\loc}(\Omega)$ is a \emph{function of locally bounded $X$-variation}, and we write $u \in \bv_{X,\loc}(\Omega)$, if there exists a $\R^m$-valued Radon measure $D_X u=(D_{X_1}u,\dots D_{X_m}u)$ on $\Omega$ such that, for every open set $A \subset \subset \Omega$, for every $1 \leq i \leq m$ and for every $\vp \in C^1_c(A)$ one has 
\[
\int_A \vp d(D_{X_i}u)=-\int_A u X_i^*\vp d \leb^n.
 \]
Moreover, if $u \in L^1(\Omega)$ and $D_X u$ has bounded total variation $|D_X u|$, then we say that $u$ has \emph{bounded $X$-variation} and we write $u \in \bv_X(\Omega)$. For every $u \in \bv_X(\Omega)$ we define the norm
\[
\nl u \nr_{\bv_X(\Omega)} \ceq \nl u \nr_{L^1(\Omega)}+|D_Xu|(\Omega).
\]
The space $\bv_X(\Omega)$ equipped with the above norm is a Banach space.
\end{defi}

\begin{defi}\label{def_xper}
    We say that a measurable set  $E \se \R^n$ has \emph{locally finite $X$-perimeter} (respectively, \emph{finite $X$-perimeter}) in $\Omega$ if its characteristic function $\chi_E$ belongs to $\bv_{X,\loc}(\Omega)$ (respectively, $\chi_E \in \bv_X(\Omega)$). 
    In such a case we define the \emph{$X$-perimeter measure} $P^X_E$ of $E$ as $P^X_E:=|D_X \chi_E|$. In the following we will sometimes use the notation $P_X(E,\cdot)$ instead of $P^X_E(\cdot)$.
\end{defi}

\begin{defi}\label{def_pror}
      We say that $(\R^n,X)$ satisfies the \emph{property $\mathcal{R}$} if for every open set $\Omega \se \R^n$ and every $E \se \R^n$ with locally finite $X$-perimeter in $\Omega$, the essential boundary $\partial^* E \cap \Omega$ is countably $X$-rectifiable. Let us recall for completeness that the \emph{essential boundary} of a measurable set $E$ is $\partial^*E \ceq \R^n \setminus (E_0 \cup E_1)$, where for $\lambda \in [0,1]$ we denote by $E_\lambda$ the set of points $p\in\R^n$ where $E$ has density $\lambda$, i.e., 
    \[
    \lim_{r \to 0}\frac{\leb^n(E \cap B(p,r))}{\leb^n(B(p,r))}=\lambda.
    \]
    \end{defi}

\begin{defi}\label{def_applim}
    Let $u \in L^1_{\loc}(\Omega)$, $z \in \R$ and $p \in \Omega$. We say that $z$ is the \emph{approximate limit} of $u$ at $p$ if
    \[
    \lim_{r \to 0} \frac{1}{\leb^n(B(p,r))} \int_{B(p,r)}|u-z|d\leb^n=0.
    \]
If the approximate limit of $u$ at $p$ exists, it is also unique (see \cite[Definition 2.19]{dv}). We hence denote by $u^\star(p)$ the approximate limit of $u$ at $p$ and by $\mathcal{S}_u$ the subset of points in $\Omega$ where $u$ does not admit an approximate limit. 
\end{defi}

\begin{defi}\label{def_c1x}
Let $\Omega \se (\R^n,X)$ be an open set and $f\colon \Omega \to \R$. We say that $f \in C^1_X(\Omega)$ if $f$ is continuous and its \emph{horizontal gradient}  $Xf \ceq (X_1f,\dots,X_mf)$,   in the sense of distributions, is represented by a continuous function.
\end{defi}

\begin{defi}\label{def_approxdiff}
    Let $u \in L^1_{\loc}(\Omega)$ and $p \in \Omega \setminus \mathcal{S}_u$. We say that $u$ is \emph{approximately $X$-differentiable} at $p$ if there exist a neighbourhood $U\subset \Omega$ of $p$ and $f \in C^1_X(U)$ such that $f(p)=0$ and
    \[
    \lim_{r \to 0} \frac{1}{\leb^n(B(p,r))} \int_{B(p,r)} \frac{|u-u^\star (p)-f|}{r}d \leb^n =0.
    \]
The vector $Xf(p) \in \R^m$ is uniquely determined (see~\cite[Proposition 2.30]{dv}): we call it \emph{approximate $X$-gradient} of $u$ at $p$ and we denote it by $D^\ap_X u(p)$.
\end{defi}

\begin{defi}\label{def_ipersup}
We say that $S \se (\R^n,X)$ is a \emph{$C^1_X$-hypersurface} if for every $p \in S$ there exist $r>0$ and $f \in C^1_X(B(p,r))$ such that the following facts hold:
\begin{enumerate}
\item[(i)]$ S \cap B(p,r)=\lbrace q \in B(p,r):f(q)=0 \rbrace$,
\item[(ii)]$Xf\neq 0$ on $B(p,r)$.
\end{enumerate}
We define the \emph{horizontal normal} to $S$ at $p \in S$ as
\[
\nu_S(p) \ceq \frac{Xf(p)}{|Xf(p)|}.
\]
Notice that $\nu_S(p)$ is well defined up to a sign and, in particular, it does not depend on the choice of $f$, see \cite[Corollary 2.14]{dv}. We say that $S  \se (\R^n,X)$ is \emph{countably $X$-rectifiable} if there exists a family $\lbrace S_h: h \in \N \rbrace$ of $C^1_X$-hypersurfaces such that
\[
\hau^{Q-1}\left( S \setminus \bigcup_{h \in \N}S_h \right)=0.
\]
  Moreover, if $\hau^{Q-1}(S)<+\infty$, we say that $S$ is \emph{$X$-rectifiable}. We define the \emph{horizontal normal} of a countably $X$-rectifiable set $S$ at $p \in S$ as 
\[
\nu_S(p) \ceq \nu_{S_h}(p) \text{ if }p \in S_h \setminus \bigcup_{k<h}S_k.
\]
Notice that $\nu_S$ is well defined, up to a sign, $\hau^{Q-1}$-a.e., see \cite[Proposition 2.18]{dv}.
\end{defi}

\begin{defi}\label{def_approxjump}
Fix $p \in (\R^n,X)$, $R>0$ and $\nu \in \mathbb{S}^{m-1}$. Let $f \in C^1_X(B(p,R))$ be such that $f(p)=0$ and $\frac{Xf(p)}{|Xf(p)|}=\nu$. For every $r \in (0,R)$ we set
\[
B^+_\nu(p,r) \ceq B(p,r) \cap \lbrace f>0\rbrace,\qquad
B^-_\nu(p,r) \ceq B(p,r) \cap \lbrace f<0\rbrace.
\]
Now let $u \in L^1_{\loc}(\Omega)$ and $p \in \Omega$. We say that $u$ has an \emph{approximate $X$-jump} at $p$ if there exist $u^+,u^- \in \R$ with $u^+ \neq u^-$ and $\nu \in \mathbb{S}^{m-1}$ such that
\begin{equation}\label{eq_xjump}
\lim_{r \to 0}\frac{1}{\leb^n(B^+_\nu(p,r))}  \int_{B^+_\nu(p,r)} |u-u^+|d\leb^n=\lim_{r \to 0} \frac{1}{\leb^n(B^-_\nu(p,r))}\int_{B^-_\nu(p,r)} |u-u^-|d\leb^n=0.
\end{equation}
The \emph{jump set} $\J_u$ is defined as the set of points where $u$ has an approximate $X$-jump. Notice that condition \eqref{eq_xjump} does not depend on the choice of the function $f$ used to construct the sets  ${B^+_\nu(p,r)}$ and ${B^-_\nu(p,r)}$, see \cite[Proposition 2.26 and Remark 2.27]{dv}.
\end{defi}

\begin{rem}\label{rem_proprietaR}
Let us stress that since we are assuming that property $\mathcal{R}$ holds we have, by \cite[Theorem 1.5]{dv}, that the discontinuity set $\s_u$ (and hence also the jump set $\J_u$) is countably $X$-rectifiable for every $u \in \bv_X$. Moreover, $\hau^{Q-1}(\s_u \setminus \J_u)=0$, see \cite[Theorem 1.5]{dv}.
\end{rem}

\begin{defi}
    Let $u \in \bv_X(\Omega)$. We define $\hau^{Q-1}$-a.e. the \emph{precise representative} $u^\p$ of $u$ as
    \[
    u^\p \ceq \begin{cases}
        u^\star \text{ on } \Omega \setminus \s_u\\
       \displaystyle \frac{u^++u^-}{2} \text{ on } \J_u,        
    \end{cases}
    \]
    where $u^\pm$ are as in Definition \ref{def_approxjump}. Moreover, by \cite[Theorem 3.14]{dv}, one has 
   \[
 \lim_{r \to 0} \frac{1}{\leb^n(B(p,r))}  \int_{B(p,r)} u d\leb^n = u^\p(p) \qquad \text{ for $\hau^{Q-1}$-a.e. $p \in \Omega$}.
   \]
    \end{defi}
\begin{defi} 
For every $u \in \bv_X(\Omega)$ we decompose
\[
D_Xu=D_X^au+D_X^su
\]
where $D_X^au$ denotes the \emph{absolutely continuous part} of $D_Xu$ (with respect to the usual Lebesgue measure $\leb^n$) and $D_X^su$ denotes the \emph{singular part} of $D_Xu$. We define the \emph{jump part} of $D_Xu$ as 
\[
D_X^ju \ceq D_X^s u \res \J_u
\]
and the \emph{Cantor part} of $D_Xu$ as
\[
D_X^cu \ceq D_X^s u \res (\Omega \setminus \J_u).
\]

 We say that $u$ is a \emph{special function of bounded $X$-variation}\footnote{Let us observe that the definition is different from the one provided in \cite{sbvx} since, by assuming property $\mathcal{R}$, the jump set is always countably $X$-rectifiable.}, and we write $u \in \sbv_X(\Omega)$, if $D_X^cu=0$.
\end{defi}

We conclude this section by recalling two results that will be pivotal in the sequel: the coarea formula and the isoperimetric inequality. From now on we will make use of the following notation: given a function $u:O \to \R$ and $t \in \R$ we will use the compact notation $\{u>t\}$, $\{u=t\}$, etc., to denote the sets $\{p \in O:u(p)>t\}$, $\{p \in O:u(p)=t\}$, etc.

\begin{teo}[{\cite[Theorem 2.3.5]{fssc}}]\label{teo_coarea1}
Let $u \in \bv_X(\Omega)$. Then $\{u>t\}$ has finite $X$-perimeter for $\leb^1$-a.e. $t \in \R$ and
\[
|D_Xu|(\Omega)=\int_{-\infty}^{+\infty} P_X(\{u >t\},\Omega)dt.
\]
Conversely, if $u \in L^1(\Omega)$ and 
\[
\int_{-\infty}^{+\infty} P_X(\{u >t\},\Omega)dt<+\infty,
\]
then $u \in \bv_X(\Omega)$.
\end{teo}
Let us observe that also the following version of the coarea formula holds. Although the proof is standard, we provide it for the sake of completeness.
\begin{teo}\label{teo_coarea2}
    Let $u \in \bv_X(\Omega)$. Then for every Borel set $A \se \Omega$ we have
\[
D_Xu(A)=\int_{-\infty}^{+\infty} D_X \chi_{\{u >t\}}(A) dt.
\]
\end{teo}
\begin{proof}
    Let $\vp \in C^1_c(\Omega)$ and $i \in \{1,\dots,m\}$. We have
    \begin{align*}
        \int_\Omega \vp dD_{X_i}u&=-\int_\Omega u(x)X^*_i\vp(x)dx\\
        &=-\int_\Omega \left( \int_0^{+\infty} \chi_{\{u>t\}}(x)dt \right)X^*_i\vp(x)dx+\int_\Omega \left( \int_{-\infty}^0 (1- \chi_{\{u>t\}})(x)dt \right)X^*_i\vp(x)dx \\
        &=-\int_0^{+\infty} \left( \int_\Omega \chi_{\{u>t\}}(x)X^*_i\vp(x)dx \right)dt+\int_{-\infty}^0 \left(\int_\Omega  (1- \chi_{\{u>t\}})(x)X^*_i\vp(x)dx \right)dt\\
        &= \int_{-\infty}^{+\infty} \left( \int_\Omega \vp dD_{X_i} \chi_{\{ u>t \} } \right)dt,
    \end{align*}
    where we used Fubini's theorem and the fact that, for every positive Borel function $v$, we have $v(x)=\int_0^{+\infty} \chi_{ \{ v>t \} }(x)dt$.
\end{proof}

For a proof of the following isoperimetric inequality in Carnot-Carathéodory spaces see for instance \cite[Theorem 2.42]{dv} and the references therein.
\begin{teo}\label{teo_isoineq}
    Let $K \se (\R^n,X)$ be a compact set. Then there exists $C_K>0$ and $R_K>0$ such that, for every $p \in K$, $r \in (0,R_K)$ and every $\leb^n$-measurable set $E$, one has
    \[
    \min \{ \leb^n(E \cap B(p,r)),\leb^n(B(p,r) \setminus E)    \}^\frac{Q-1}{Q} \leq C_KP_X(E,B(p,r)).
    \]
\end{teo}

\section{Intrinsic perimeter measure of countably \texorpdfstring{$X$}{X}-rectifiable sets}\label{sec_intrper}

\begin{defi}\label{def_pxu}
Let $R \se (\R^n,X)$ be a countably $X$-rectifiable set. By definition it can be covered, up to a $\hau^{Q-1}$-negligible set, by countably many $C^1_X$-hypersurfaces $(S_j)_{j \in \N}$, that we may assume to be pairwise disjoint. Locally, each hypersurface $S_j$ separates the space $(\R^n,X)$ into two  $X$-regular open sets (see~\cite{vittone2012}). We denote by $P_j^X$ the measure on $S_j$ defined locally as the  $X$-perimeter measure of (each of) these two components. By the representation of the $X$-perimeter measure (see \cite{AmbrosioAhlfors}) we have, for a suitable positive function $\theta_j$ locally bounded away from 0, that
\[
P_j^X \res R = \theta_j \hau^{Q-1} \res S_j.
\]
We define $\hau^{Q-1}$-a.e. on $R$ the map $\theta_R$ as $\theta_R|_{S_j} \ceq \theta_j$. Finally we define the \emph{intrinsic perimeter measure of $R$} as
\[
\mathcal{P}_R^X \ceq \theta_R \hau^{Q-1} \res R.
\]

Let us observe that $\mathcal P^X_R$ is unique: in fact, it is easy to check that its definition does not depend on the choice of the collection of $C^1_X$-hypersurfaces $(S_j)_{j \in \N}$.
\end{defi}

Using the $X$-perimeter measure just introduced in Definition \ref{def_pxu} and \cite[Theorem 1.1]{dv} we can rewrite \cite[Proposition 3.3]{sbvx} as follows.
\begin{pro}\label{prop_defequivsbvx}
    The following statements are equivalent:
    \begin{enumerate}
        \item[(i)] $u\in \sbv_X(\Omega)$;
        \item[(ii)] $u\in \bv_X(\Omega)$ and  
        \begin{equation}\label{eq_rappsbv}
        D_Xu= D^\ap_Xu \leb^n+(u^+-u^-)\nu_{\J_u}\pxu
        \end{equation}
    \end{enumerate}
    where $u^\pm$ are as in Definition \ref{def_approxjump}.
\end{pro}
Let us stress that the measure $\pxu$ which appears in \eqref{eq_rappsbv} does not depend on the function $u$ but only on the countably $X$-rectifiable set $\J_u$.

By recalling that the absolutely continuous part, the jump part and the Cantor part of $D_Xu$ are mutually singular we also obtain the following version of \cite[Theorem 1.3 (i)]{dv}.

\begin{teo}
	For every $u\in \bv_X(\Omega)$ we have $|D_Xu|\geq  |u^+-u^-|\pxu$, where $u^\pm$ are as in Definition \ref{def_approxjump}.
\end{teo}

\section{Compactness of \texorpdfstring{$\sbv_X$}{SBVX} functions}\label{sec_compactness}

\begin{lem}\label{lemma_chain}
    Let $u \in \bv_X(\Omega)$. Let $\phi \in C^1(\R) \cap W^{1,\infty}(\R)$. Then $ \phi \circ u\in \bv_X(\Omega)$ and
    \begin{equation}\label{eq_chainrule}
    D_X(\phi \circ u) \res (\Omega \setminus \s_u) =\phi'(u^\p) D_Xu \res (\Omega \setminus \s_u).
    \end{equation}
\end{lem}
\begin{proof}
    Let us first prove that $\phi \circ u \in \bv_X(\Omega)$. By \cite[Theorem 2.2.2]{fssc} there exists a sequence of functions $(u_k)_{k \in \N}$ such that, for every $k \in \N$, we have $u_k \in C^\infty(\Omega) \cap \bv_X(\Omega)$ and
    \[
    \nl u-u_k \nr_{L^1(\Omega)} \xrightarrow{k \to +\infty}0, \qquad |D_Xu_k|(\Omega) \xrightarrow{k \to +\infty}|D_Xu|(\Omega).
    \]
    Then $\nl \phi \circ u-\phi \circ u_k \nr_{L^1(\Omega)} \xrightarrow{k \to +\infty}0$ and, by the lower semicontinuity of the total variation, 
    \[
    |D_X(\phi \circ u)|(\Omega) \leq \liminf_{k \to +\infty} |D_X(\phi \circ u_k)|(\Omega) \leq \nl \phi' \nr_{\infty}\liminf_{k \to +\infty}|D_Xu_k|(\Omega)=\nl \phi' \nr_{\infty}|D_Xu|(\Omega)<\infty,
    \]
concluding $\phi \circ u \in \bv_X(\Omega)$. Since every $\phi \in C^1(\R) \cap W^{1,\infty}(\R)$ can be written as the difference of two strictly increasing functions with the same properties, in the rest of the proof we assume, without loss of generality\footnote{By the linearity of \eqref{eq_chainrule}.}, that $\phi$ is also strictly increasing and, in particular, injective. By the Lipschitzianity of $\phi$ we have $\s_{\phi \circ u} \se \s_u$. Let $A \se \Omega \setminus \s_u$ be any Borel set. By the coarea formula (Theorem \ref{teo_coarea2}) we have
\[
D_X(\phi \circ u)(A)=\int_{-\infty}^{+\infty} D_X\chi_{  \{ \phi \circ u >s  \}  } (A)ds.
\]
By the change of variable $s=\phi(t)$ we obtain
\[
\int_{-\infty}^{+\infty} D_X\chi_{\{\phi \circ u>s \}} (A)ds=\int_{-\infty}^{+\infty} \phi'(t)D_X\chi_{\{u>t \}} (A)dt.
\]
We have that if $x \in A \cap \pp^* \{u>t \} $, then $u^\p(x)=t$ (see for instance \cite[Proposition 2.22]{dv}) and 
\[
\int_{-\infty}^{+\infty} \phi'(t)D_X\chi_{\{u>t \}} (A)dt=\int_{-\infty}^{+\infty}\int_A\phi'(u^\p(x))dD_X\chi_{\{u>t \}}(x)dt
\]
Using again the coarea formula (Theorem \ref{teo_coarea2}) and Fubini's theorem we obtain
\[
\int_{-\infty}^{+\infty}\int_A\phi'(u^\p(x))dD_X\chi_{\{u>t \}}(x)dt=\int_A \phi'(u^\p)dD_Xu
\]
the latter implying that 
\[
D_X(\phi \circ u) \res (\Omega \setminus \s_u)=\phi'(u^\p)D_Xu\res (\Omega \setminus \s_u),   
\]
concluding the proof.

\end{proof}

\begin{lem}\label{lemma_chain2}
     Let $u \in \sbv_X(\Omega)$. Let $\phi \in C^1(\R) \cap W^{1,\infty}(\R)$. Then $\phi \circ u \in \sbv_X(\Omega)$ and
    \begin{equation}\label{eq_chain2}
    D_X(\phi \circ u)=\phi'(u^\p)D^\ap_Xu\leb^n+(\phi(u^+)-\phi(u^-))\nu_{\J_u}\mathcal{P}^X_{\J_u}.
    \end{equation}
\end{lem}
\begin{proof}
We have   \[
    D_X(\phi \circ u)=D_X(\phi \circ u)\res( \Omega \setminus \s_u)+D_X(\phi \circ u) \res (\s_u \setminus \J_u)+D_X(\phi \circ u) \res \J_u.
    \]
    By Lemma \ref{lemma_chain}, we have $\phi \circ u \in \bv_X(\Omega)$ and
    \begin{equation}\label{eq_aa1}
    D_X(\phi \circ u) \res (\Omega \setminus \s_u) =\phi'(u^\p) D_Xu \res (\Omega \setminus \s_u)= \phi'(u^\p)D^\ap_Xu\leb^n
    \end{equation}
    By \cite[Theorem 1.5]{dv} we know $\hau^{Q-1}(\s_u \setminus \J_u )=0$ and by \cite[Theorem 1.3 (ii)]{dv} we obtain 
        \begin{equation}\label{eq_aa2}
    D_X(\phi \circ u) \res (\s_u \setminus \J_u)=0.
  \end{equation}
    We have, by \cite[Proposition 2.28 (ii)]{dv}, that $\J_{\phi \circ u} \se \J_u$ and, on $\J_u$, we have $\phi(u^+) \equiv (\phi \circ u)^+$ and $\phi(u^-) \equiv (\phi \circ u)^-$.  In the same fashion as in Lemma \ref{lemma_chain}, we can write $\phi$ as a difference of two strictly increasing functions in $C^1(\R) \cap W^{1,\infty}(\R)$, hence, by the linearity of \eqref{eq_chain2}, we may assume in the rest of this proof that $\phi$ is strictly increasing and, in particular, injective. As a consequence, we have $\J_{\phi \circ u}=\J_u$, the latter implying
        \begin{equation}\label{eq_aaaa}
    D_X(\phi \circ u)\res \J_u =D^j_X (\phi \circ u).
    \end{equation}   
 By \eqref{eq_aa1}, \eqref{eq_aa2} and \eqref{eq_aaaa} we have $D^c_X(\phi \circ u) =0$, hence $\phi \circ u \in \sbv_X(\Omega)$. Finally, by Proposition \ref{prop_defequivsbvx}, we obtain 
   \begin{equation}\label{eq_aa3}
   D^j_X (\phi \circ u)=(\phi(u^+)-\phi(u^-))\nu_{\J_u}\pxu.
   \end{equation}
 Putting together \eqref{eq_aa1}, \eqref{eq_aa2} and \eqref{eq_aa3} we conclude the proof.
 \end{proof}

\begin{defi}\label{def_gammaf}
    Let $f:[0,+\infty) \to [0,+\infty]$ be an increasing function such that 
    \[
    \lim_{t \to 0}\frac{f(t)}{t}=+\infty,
    \]
    and define $\Gamma(f)$ as the class of all $C^1$ and $W^{1,\infty}$ functions $\phi:\R \to \R$ which satisfy
    \[
    |\phi(t')-\phi(t)|\leq f(|t'-t|) \qquad \text{ for all } t,t' \in \R.
    \]
\end{defi}

\begin{rem}\label{rem_ineq}
    We observe in passing that if $u \in \sbv_X(\Omega)$ and $\phi \in \Gamma(f)$ where $f,\Gamma(f)$ are as in Definition \ref{def_gammaf} above one has from Lemma \ref{lemma_chain2} that
    \[
    \nl D_X (\phi \circ u)-\phi'(u^\p)D^\ap_Xu\leb^n  \nr\leq \int_{\J_u}|\phi(u^+)-\phi(u^-)|d \pxu\leq \int_{\J_u} f(|u^+-u^-|)d\pxu.
    \]
\end{rem}

\begin{pro}\label{pro_aux}
    Let $f,\Gamma(f)$ be given as in Definition \ref{def_gammaf}. Let $u \in \bv_X(\Omega)$ and let $g \in L^1(\Omega,\R^m)$ such that
    \[
    \sup_{\phi \in \Gamma(f)}\nl D_X(\phi \circ u)-\phi'(u^\p)g \leb^n \nr <+\infty.
    \]
        Then $g =D_X^\ap u$ and $u \in \sbv_X(\Omega)$.
\end{pro}
\begin{proof}
Define $\mu \ceq (D_X^\ap u-g) \leb^n+D^c_Xu$. We have to prove that $\mu=0$. By using the chain rule \eqref{eq_chainrule} we obtain
\begin{align*}
\sup_{\phi \in \Gamma(f)} &\nl      D_X(\phi \circ u) \res (\Omega \setminus \s_u)+D_X(\phi \circ u) \res (\s_u \setminus \J_u)  +D_X(\phi \circ u) \res \J_u-\phi'(u^\p)g\leb^n       \nr\\
&=\sup_{\phi \in \Gamma(f)}\nl      \phi'(u^\p)[D_X u \res (\Omega \setminus \s_u)-g\leb^n] +D_X^j(\phi \circ u)      \nr\\
& \geq \sup_{\phi \in \Gamma(f)}\nl   \phi'(u^\p) [D_X u \res (\Omega \setminus \s_u)-g\leb^n]   \nr\\
&  = \sup_{\phi \in \Gamma(f)}\nl \phi'(u^\p)   [(D^\ap_Xu-g)\leb^n+D^c_Xu ]           \nr\\
&=\sup_{\phi \in \Gamma(f)}\nl \phi'(u^\p)   \mu\nr.
\end{align*}
By hypothesis we obtain
\[
\sup_{\phi \in \Gamma(f)}\int_\Omega|\phi'(u^\p)|d|\mu|<+\infty.
\]
Now, if we denote by $\sigma$ the positive measure on $\R$ given by $\sigma(B)=|\mu|((u^\p)^{-1}(B))$ for all Borel sets $B \se \R$, we obtain
\[
+\infty>\sup_{\phi \in \Gamma(f)}\int_\Omega|\phi'(u^\p)|d|\mu|=\sup_{\phi \in \Gamma(f)}\int_\R |\phi'|d\sigma.
\]
By \cite[Lemma 2.4]{albertimantegazza} we obtain $\sigma=\mu=0$, concluding the proof.
\end{proof}

\begin{teo}\label{teo_prec}
 Let $f$ be as in Definition \ref{def_gammaf} and $\Omega \se (\R^n,X)$ be an open bounded set. Let $(u_h)_{h \in \N}$ be a sequence of $\sbv_X(\Omega)$ functions which converges to $u$ in the weak-$*$ topology of $\bv_X(\Omega)$, and assume the approximate $X$-gradient $D^\ap_Xu_h$ converge to some $g$ weakly in $L^1(\Omega,\R^m)$ and that 
    \[
    \int_{\J_{u_h}}f(|u^+_h-u^-_h|)d\mathcal{P}^X_{\J_{u_h}} \leq C <+\infty \qquad \text{ for every }h \in \N.
    \]
 Then $u \in \sbv_X(\Omega)$, and the absolutely continuous part and the jump part of the derivatives converge separately, i.e.,  $D^\ap_Xu_h \leb^n \xrightarrow{h \to +\infty}D^\ap_Xu\leb^n$ and $D_X^ju_h \xrightarrow{h \to +\infty}D_X^ju$ in the weak-$*$ topology of measures.

\begin{proof}
    It is enough to prove that $D^\ap_Xu \leb^n+D^c_Xu=g \leb^n$. Let $\phi \in \Gamma(f)$. Since $u_h \in \sbv_X(\Omega)$ for every $h \in \N$ we have, by Remark \ref{rem_ineq},
    \begin{equation}\label{ex_aux}
    C \geq \int_{\J_{u_h}}f(|u_h^+-u^-_h|)d \mathcal{P}^X_{\J_{u_h}} \geq \nl D_X(\phi \circ u_h)-\phi'(u_h^\p)D^\ap_Xu_h\leb^n \nr.
    \end{equation}
Without loss of generality we may assume that $u_h$ converge to $u$ almost everywhere in $\Omega$, the latter implying that the functions $\phi(u_h)$ converge to $\phi(u)$ in the weak-$*$ topology of $\bv_X(\Omega)$, and the measures $D_X(\phi \circ u_h)$ converge to $D_X(\phi \circ u)$ in the weak-$*$ topology of measures. Since $\phi'$ is bounded and continuous, the functions $\phi'(u_h)$ are uniformly bounded and converge to $\phi'(u)$ a.e.. Moreover the functions $D^\ap_Xu_h$ converge to $g$ weakly in $L^1(\Omega,\R^m)$ by hypothesis. Hence $\phi'(u_h)D^\ap_Xu_h$ converge to $\phi'(u)g$ weakly in $L^1(\Omega,\R^m)$ and then
\[
D_X(\phi \circ u_h)-\phi'(u_h)D^\ap_Xu_h \leb^n \xrightarrow{h \to +\infty}D_X(\phi \circ u)-\phi'(u)g\leb^n
\]
in the weak-$*$ topology of measures. By \eqref{ex_aux} we obtain
\[
C \geq \liminf_{h \to +\infty}\nl D_X(\phi \circ u_h)-\phi'(u_h)D^\ap_Xu_h\leb^n \nr \geq \nl D_X(\phi \circ  u)-\phi'(u)g\leb^n \nr.
\]
Taking the supremum over all $\phi \in \Gamma(f)$ and using Proposition \ref{pro_aux} we conclude that $u \in \sbv_X(\Omega)$.
\end{proof}
\end{teo}
As an immediate consequence of Theorem \ref{teo_prec} we get the compactness theorem for $\sbv_X$ functions, Theorem \ref{teo_introcompact}.

\section{Poincaré inequality for \texorpdfstring{$\sbv_X$}{SBVX} functions}
\label{sec_poinc}
\begin{pro}\label{prop_maxmin}

    Let $u \in \bv_X(O)$ where $O$ is a bounded open set and $\lambda \in \R$. Then $\max\{u,\lambda \}, \min\{u,\lambda \} \in \bv_X(O)$ and
      \begin{align*}
    D_X \max \{u,\lambda\}&= (D^\ap_Xu\leb^n+D^c_Xu) \res \{u^\p \geq \lambda \}+ (\max\{u^+,\lambda \}-\max\{u^-,\lambda \} )\nu_{\J_u}\pxu ,\\
       D_X \min \{u,\lambda\}&= (D^\ap_Xu\leb^n+D^c_Xu)  \res \{u^\p \leq \lambda \}+(\min\{u^+,\lambda \}-\min\{u^-,\lambda \} )\nu_{\J_u}\pxu.
    \end{align*}
\end{pro}
\begin{proof}

Define $\phi:\R \to \R$ as $\phi(t) \ceq \max\{t,\lambda\}$. Clearly, since $O$ is bounded, $\phi \circ u \in L^1(O)$. We observe, using also the coarea formula (Theorem \ref{teo_coarea1}), that for every Borel set $A \se O$,
\[
\int_{-\infty}^{+\infty} P_X(\{\phi \circ u>s \},A)ds=\underbrace{\int_{-\infty}^{\lambda}P_X(\{\phi \circ u>s \},A)ds}_{=0}+\int_{\lambda}^{+\infty}P_X(\{ u>s \},A)ds \leq |D_Xu|(A)
\]
which, thanks to Theorem \ref{teo_coarea1}, implies that $\max\{u,\lambda\} \in  \bv_X(O)$.  By the Lipschitzianity of $\phi$ we have $\s_{\phi \circ u} \se \s_u$. Let $A \se O \setminus \s_u$ be any Borel set. By the coarea formula (Theorem \ref{teo_coarea2}) we have
\begin{align*}
 D_X \max \{u,\lambda\}(A)&=   \int_{-\infty}^{+\infty}D_X \chi_{\{\phi \circ u>s \}}(A)ds=  \underbrace{  \int_{-\infty}^{\lambda}D_X \chi_{\{\phi \circ u>s \}}(A)ds}_{=0}+    \int_{\lambda}^{+\infty}D_X \chi_{\{\phi \circ u>s \}}(A)ds\\
    &=\int_\lambda^{+\infty} D_X \chi_{\{ u>s \}}(A)ds\\
    &=\int_{-\infty}^{+\infty} \chi_{[\lambda,+\infty]}(s) D_X \chi_{\{u>s \}}(A)ds
    \end{align*}
We have that if $x \in A \cap \pp^*\{ u>s \}$, then $u^\p(x)=s$ (see for instance \cite[Proposition 2.22]{dv}) and 
\[
\int_{-\infty}^{+\infty}\chi_{[\lambda,+\infty]}(s)D_X\chi_{\{ u>s \}} (A)ds=\int_{-\infty}^{+\infty}\int_A\chi_{[\lambda,+\infty]}(u^\p(x))dD_X\chi_{\{ u>s \}}(x)ds
\]
Using again the coarea formula and Fubini's theorem we obtain
\[
\int_{-\infty}^{+\infty}\int_A\chi_{[\lambda,+\infty]}(u^\p(x))dD_X\chi_{\{ u>s \}}(x)ds=\int_A \chi_{[\lambda,+\infty]}(u^\p)dD_Xu
\]
the latter implying that 
\[
D_X \max \{u,\lambda \}  \res (O \setminus \s_u)=D_Xu\res \big((O \setminus \s_u) \cap \{u^\p \geq \lambda \} \big).  
\]
In the same fashion of Lemma \ref{lemma_chain2} one can prove that
\[
D_X^j \max \{ u,\lambda\}= (\max\{u^+,\lambda \}-\max\{u^-,\lambda \} )\nu_{\J_u}\pxu,
\]
proving the claim for $D_X \max \{ u,\lambda\}$. The analogous proof works for proving the claim for $D_X \min \{ u,\lambda\}$.

\end{proof}

\begin{rem}\label{rem_maxmin}
    As an immediate consequence of Proposition \ref{prop_maxmin} one has, for every $\lambda \in \R$, $u \in \sbv_X(O)$ with $O$ bounded open set, that $\max\{u,\lambda \}, \min\{u,\lambda \} \in \sbv_X(O)$. Let us also observe that if $u,v \in \bv_X(O)$, Proposition  \ref{prop_maxmin} allows to conclude that  $\max\{ u,v \},\min \{u,v\} \in \bv_X(O) $ since it suffices to write $\max \{ u,v \}=u+\max \{v-u,0\}$ and $\min \{ u,v \}=u+\min \{v-u,0\}$.
\end{rem}

\begin{defi} 
Let $K \se (\R^n,X)$ be a compact set and $C_K$,$R_K$ be the constants defined as in Theorem \ref{teo_isoineq}. Fix $p \in K$, $r \in (0,R_K)$ 
and define $B \ceq B(p,r)$. Let $u:B \to \R$ be a measurable function and define the \emph{non-decreasing rearrangement} of $u$
\[
u_*(s) \ceq \inf \{ t \in [-\infty,+\infty]: \leb^n(\{u<t\}) \geq s \}
\]
for any $s \in [0,\leb^n(B)]$. If $m \in \R$ is such that
\[
\leb^n( \{u<t\}) \leq \frac{\leb^n(B)}{2} \quad \text{ for all }t<m, \qquad \leb^n( \{u>t\}) \leq \frac{\leb^n(B)}{2} \quad \text{ for all }t>m,
\]
 we say that $m$ is a \emph{median} for $u$ in $B$. For every $u \in \sbv_X(B)$ such that 
    \begin{equation}\label{eq_condhj}
        \pxu(\J_u)<\frac{1}{2C_K}\left(\frac{1}{2}\leb^n(B)  \right)^{\frac{Q-1}{Q}},
    \end{equation} 
we set
\begin{equation}\label{eq_tau1}
    \tau_1(u) \ceq u_* \bigg( \big(2  C_K \pxu(\J_u)\big)^\frac{Q}{Q-1} \bigg),
\end{equation}
    \begin{equation}\label{eq_tau2}
    \tau_2(u) \ceq u_* \bigg( \leb^n(B)- \big(2C_K \pxu(\J_u)\big)^\frac{Q}{Q-1} \bigg).
\end{equation}
Clearly, for every median $m$ for $u$ in $B$ we have
\[
\tau_1(u) \leq m \leq \tau_2(u).
\]
\end{defi}

In this section we will also need the following version of Poincaré inequality for $\bv_X$ function. 
\begin{teo}\label{teo_poincarebv}
    Let $K \se (\R^n,X)$ be a compact set. Then there exists $C_K>0$ and $R_K>0$ such that, for every $p \in K$, for  every $r \in (0,R_K)$ and for every $u \in \bv_{X}(B(p,r))$ one has
    \[
  \left(  \int_{B(p,r)} |u-m|^\frac{Q}{Q-1}d\leb^n \right)^\frac{Q-1}{Q}\leq C_K|D_Xu|(B(p,r)),
    \]
    where $m$ is any median for $u$ in $B(p,r)$.
\end{teo}

\begin{proof}
    Up to a translation, we may assume $m=0$. We write 
    \[
    u=\underbrace{\max\{u,0\}}_{u_1}-\underbrace{\max\{-u,0\}}_{u_2}.
    \]
    By Proposition  \ref{prop_maxmin} we  have that $u_1,u_2 \in \bv_X(B(p,r))$.  By \cite[Proposition 1.78]{afp} and by making the change of variable $s=t^\frac{Q-1}{Q}$ we have
    \begin{align*}
    \int_{B(p,r)} u_1^\frac{Q}{Q-1}d\leb^n&= \int_0^{+\infty} \leb^n (\{u_1^\frac{Q}{Q-1}>t\})dt=\int_0^{+\infty}  \leb^n (\{u_1>t^\frac{Q-1}{Q}\})dt\\&=\frac{Q}{Q-1}\int_0^{+\infty} \leb^n(\{u_1>s\})s^\frac{1}{Q-1}ds.
    \end{align*}
   By \cite[Lemma 3.48]{afp} we obtain
   \[
   \frac{Q}{Q-1}\int_0^{+\infty} \leb^n(\{u_1>s\})s^\frac{1}{Q-1}ds \leq \left(   \int_0^{+\infty} \leb^n(\{u_1>s\})^\frac{Q-1}{Q} ds\right)^\frac{Q}{Q-1}.
   \]
Since $0$ is a median also for $u_1$ (and $u_2$) by the isoperimetric inequality (Theorem \ref{teo_isoineq})
\begin{align*}
\left( \int_0^{+\infty} \leb^n(\{u_1>s\})^\frac{Q-1}{Q} ds\right)^\frac{Q}{Q-1} &\leq \left( \int_0^{+\infty} C_KP_X(\{u_1>s\},B(p,r)) ds      \right)^\frac{Q}{Q-1}
\end{align*}
obtaining 
\[
  \left(  \int_{B(p,r)} u_1^\frac{Q}{Q-1}d\leb^n \right)^\frac{Q-1}{Q}\leq  C_K\int_0^{+\infty}P_X(\{u_1>s\},B(p,r)) ds   
\]
In the same fashion we obtain the analogous inequality for $u_2$ and, consequently,
\begin{align*}
\nl u \nr_{L^\frac{Q}{Q-1}(B(p,r))} &\leq \nl u_1 \nr_{L^\frac{Q}{Q-1}(B(p,r))}+\nl u_2 \nr_{L^\frac{Q}{Q-1} (B(p,r))}\\& \leq C_K\int_0^{+\infty}P_X(\{u_1>s\},B(p,r)) ds   +C_K\int_0^{+\infty}P_X(\{u_2>s\},B(p,r)) ds   \\&=C_K \int_{-\infty}^{+\infty}P_X(\{u>s\},B(p,r)) ds, 
\end{align*}
and we conclude by using the coarea formula (Theorem \ref{teo_coarea1}).
\end{proof}

We are ready to prove the Poincaré inequality for $\sbv_X$ functions, Theorem \ref{teo_intropoincare}, which we can now state precisely.
\begin{teo}
 Let $K \se (\R^n,X)$ be a compact set and $C_K$,$R_K$ be the constants defined as in Theorem \ref{teo_isoineq}. Fix $p \in K$, $r \in (0,R_K)$ and define $B \ceq B(p,r)$. Let $u \in \sbv_X(B)$ such that condition \eqref{eq_condhj} is satisfied. Define
    \[
    \bar u \ceq \max\{ \min\{u,\tau_2(u)\},\tau_1(u)\}  ,  
    \]
    where $\tau_1,\tau_2$ are defined as in \eqref{eq_tau1}, \eqref{eq_tau2}. Then
    \[
   \left( \int_{B} |\bar u-m|^\frac{Q}{Q-1}d\leb^n \right)^\frac{Q-1}{Q}\leq 2 C_K \int_{B}|D^\ap_Xu|d\leb^n,
    \]
    where $m$ is any median for $u$ in $B(p,r)$.
\end{teo}
\begin{proof}
    As we mentioned before, $\tau_1(u) \leq m \leq \tau_2(u)$. Up to a translation, we may assume $m=0$. By Remark \ref{rem_maxmin} $\bar u \in \sbv_X(B)$, therefore  
    \[
|D_X \bar u|(B) \leq \nl  D^\ap_X \bar u \nr_{L^1(B)}+\int_{\J_{\bar u}} | \bar u^+ - \bar u^-|d\pxu.
\]
By Proposition \ref{prop_maxmin} we get
\begin{equation}\label{eq_stimabasso}
|D_X \bar u|(B) \leq \nl  D^\ap_X  u \nr_{L^1(B)}+ (\tau_2(u)-\tau_1(u))\pxu(\J_u).
\end{equation}
By the coarea formula (Theorem \ref{teo_coarea1}) and the isoperimetric inequality (Theorem \ref{teo_isoineq}), we obtain
\begin{align}\label{eq_stimaalto}
|D_X \bar u|(B)& = \int_{\tau_1(u)}^{\tau_2(u)} P_X (\{ \bar u < t \},B)dt \\ &\geq \frac{1}{C_K} \int^0_{\tau_1(u)} \leb^n(\{ \bar u < t   \} \cap B)^\frac{Q-1}{Q}dt+\frac{1}{C_K}\int_0^{\tau_2(u)} \leb^n(B  \setminus \{ \bar u < t   \})^\frac{Q-1}{Q}dt. \nonumber
\end{align}
By the definitions of $\tau_1,\tau_2$ and \eqref{eq_condhj} we have, for $\tau_1(u)<t<0$,
\begin{equation}\label{eq_esttau1}
    \frac{1}{C_K}\leb^n(\{\bar u<t \} \cap B)^\frac{Q-1}{Q} \geq 2\pxu(\J_u),
\end{equation}
and for $0<t<\tau_2(u)$,
\begin{equation}\label{eq_esttau2}
    \frac{1}{C_K} \leb^n(B \setminus \{ \bar u <t \})^\frac{Q-1}{Q} \geq 2\pxu(\J_u).
\end{equation}
Putting \eqref{eq_esttau1} and \eqref{eq_esttau2} into \eqref{eq_stimaalto} one obtains
\begin{equation}\label{eq_stimaalto2}
    |D_X\bar u|(B) \geq 2 (\tau_2(u)-\tau_1(u))\pxu(\J_u).
\end{equation}
From \eqref{eq_stimabasso} and \eqref{eq_stimaalto2} we obtain the inequality
\[
(\tau_2(u)-\tau_1(u))\pxu(\J_u) \leq  \nl D_X^\ap u \nr_{L^1(B)},
\]
and, confronting again with \eqref{eq_stimabasso}, we obtain
\begin{equation}\label{eq_aux1}
|D_X \bar u|(B) \leq 2  \nl D_X^\ap  u \nr_{L^1(B)}.
\end{equation}
By Theorem \ref{teo_poincarebv} we have
\[
\left( \int_B | \bar u|^\frac{Q}{Q-1}d\leb^n  \right)^\frac{Q-1}{Q} \leq C_K|D_X \bar u|(B)
\]
and, combining the latter with \eqref{eq_aux1}, we get
\[
\left(  \int_B |\bar u|^\frac{Q}{Q-1} d\leb^n  \right)^\frac{Q-1}{Q} \leq 2C_K  \nl D^\ap_X u \nr_{L^1(B)},
\]
concluding the proof.
\end{proof}

\end{document}